\documentclass[12pt,reqno]{amsart}

\usepackage{amssymb,amsmath,graphicx,amsfonts,euscript}
\usepackage{color}

\newtheorem{thm}{Theorem}[section]

\newtheorem{prop}[thm]{Proposition}
\newtheorem{define}[thm]{Definition}

\newtheorem{lemma}[thm]{Lemma}

\numberwithin{equation}{section}


\begin{document}
\title[The 2D Boussinesq Equations with vertical dissipation]
{Global regularity for the 2D anisotropic Boussinesq Equations with vertical dissipation}
\author[C. Cao and J. Wu]{Chongsheng Cao$^{1}$ and Jiahong Wu$^{2}$}

\address{$^1$ Department of Mathematics,
Florida International University,
Miami, FL 33199, USA.}

\email{caoc@fiu.edu}

\address{$^2$Department of Mathematics,
Oklahoma State University,
401 Mathematical Sciences,
Stillwater, OK 74078, USA.}

\email{jiahong@math.okstate.edu}

\date{\today}

\begin{abstract}
This paper establishes the global in time existence of classical solutions to the 2D anisotropic Boussinesq equations with vertical dissipation. When only the vertical dissipation is present, there is no direct control on the horizontal derivatives and the global regularity problem is very challenging. To solve this problem, we bound the derivatives in terms of the $L^\infty$-norm of the vertical velocity $v$ and prove that $\|v\|_{L^{r}}$ with $2\le r<\infty$ at any time does not grow faster than $\sqrt{r \log r}$ as $r$ increases. A delicate interpolation inequality connecting $\|v\|_{L^\infty}$ and $\|v\|_{L^r}$ then yields the desired global regularity.

\end{abstract}
\maketitle

\section{Introduction}\label{intro}

The Boussinseq equations model many geophysical flows such as atmospheric fronts and ocean circulations (see, e.g., \cite{Maj,Pe}). Mathematically the 2D Boussinesq equations serve as a lower-dimensional model of the 3D hydrodynamics equations. In fact, the 2D Boussinesq equations retain some key features of the 3D Euler and Navier-Stokes equations such as the vortex stretching mechanism and, as pointed out in \cite{MB}, the inviscid 2D Boussinesq equations are identical to the Euler equations for the 3D axisymmetric swirling flows. The fundamental issue of whether classical solutions to the 3D Euler and Navier-Stokes equations can develop finite time singularities remains outstandingly open and the study of the 2D Boussinesq equations may shed light on this extremely challenging problem.

\vskip .1in
This paper addresses the global regularity problem concerning the 2D anisotropic Boussinesq equations with vertical dissipation,
\begin{equation}\label{vbou}
 \left\{
\begin{array}{l}
 u_t+uu_x+vu_y=-p_x+\nu\, u_{yy}, \\
 v_t+uv_x+vv_y=-p_y+\nu\, v_{yy}+\theta, \\
 u_x+v_y=0,\\
 \theta_t+u\theta_x+v\theta_y=\kappa\, \theta_{yy},\\
 u(x,y,0)=u_0(x,y), \quad v(x,y,0)=v_0(x,y), \quad \theta(x,y,0)=\theta_0(x,y),
 \end{array} \right.
\end{equation}
where $u, v, p$ and $\theta$ are scalar functions of $(x,y)\in \mathbb{R}^2$ and $t\ge 0$. Physically, $(u,v)$ denotes the 2D velocity field, $p$ the pressure, $\theta$ the temperature in the content of thermal convection and the density in the modeling of geophysical fluids, $\nu>0$ the viscosity and $\kappa>0$ the thermal diffusivity. (\ref{vbou}) is useful in modeling dynamics of geophysical flows in which the vertical dissipation dominates such as in the large-time dynamics of certain strongly stratified flows (see \cite{MG} and the references therein).

\vskip .1in
The 2D anisotropic Boussinesq system with vertical diffusion is a special case of
the general 2D Boussinesq equations
\begin{equation}\label{diss}
 \left\{
\begin{array}{l}
 u_t+uu_x+vu_y=-p_x+\nu_1\, u_{xx}+\nu_2\, u_{yy}, \\
 v_t+uv_x+vv_y=-p_y+\nu_1\, v_{xx}+\nu_2\, v_{yy}+\theta, \\
 u_x+v_y=0,\\
 \theta_t+u\theta_x+v\theta_y=\kappa_1\, \theta_{xx} + \kappa_2\, \theta_{yy},
 \end{array} \right.
 \end{equation}
where $\nu_1$, $\nu_2$, $\kappa_1$ and $\kappa_2$ are real parameters. \eqref{diss} with $\nu_1=\kappa_1=0$ reduces to \eqref{vbou}. When all four parameters are positive, \eqref{diss} is fully dissipative and the global regularity has been obtained (see, e.g., \cite{CaDi}). On the other hand, if all  parameters are zero, \eqref{diss} reduces to the inviscid Boussinesq equations. As mentioned before, the inviscid Boussinesq equations can be identified with the 3D axisymmetric Euler equations and whether or not their solutions can develop any finite-time singularity remains elusive. Several analytic and numerical results on the inviscid Boussinesq equations are available in \cite{Cha1,ES}.
The intermediate cases when at least one of four parameters in \eqref{diss} is zero has attracted considerable attention in the last few years and important progress has been made (see, e.g., \cite{AbHm,ACW10,ACW11,Ch,DP1, DP2,DP3,HmKe1,HmKe2,HKR1,HKR2,HL,LLT,MX}).  The global regularity for the case when $\nu_1 =\nu_2 >0$ and $\eta_1=\eta_2=0$ was proven by Chae \cite{Ch} and by Hou and Li \cite{HL}. The case when $\nu_1=\nu_2=0$ and $\eta_1=\eta_2>0$ was dealt with by Chae \cite{Ch}. Their results successfully resolved one of the open problems proposed by Moffatt \cite{Mof}. Further progress on these two cases was recently made  by Hmidi, Keraani and Rousset, who were able to establish the global regularity even when the full Laplacian dissipation is replaced by the critical dissipation represented in terms of the operator $\sqrt{-\Delta}$ (\cite{HKR1},\cite{HKR2}). In addition,  Miao and Xue obtained the global regularity of the 2D Boussinesq equations with fractional dissipation and thermal diffusion when the fractional powers obey certain conditions \cite{MX}. The global well-posedness for the anisotropic Boussinesq equations with horizontal dissipation or thermal diffusion, namely (\ref{diss}) with only $\nu_1>0$ or $\kappa_1>0$ was first studied by Danchin and Paicu \cite{DP3}. Recently Larios, Lunasin and Titi \cite{LLT} further investigated the Boussinesq equations with horizontal dissipation via more elementary approaches and re-established the results of Danchin and Paicu under milder assumptions. Other interesting recent results on the 2D Boussinesq equations can be found in \cite{AbHm,ACW10,ACW11,DP1,DP2, HmKe1,HmKe2,MX}.

\vskip .1in
This paper singles out the anisotropic Boussinesq equations with vertical dissipation (\ref{vbou}) for study. Why is the global regularity problem for this case difficult? Although the global (in time) $L^2$-bound follows from an easy energy estimate,  it appears impossible to directly obtain a global $H^1$-bound. There is a simple explanation. The equation satisfied by the vorticity $\omega= v_x -u_y$ is given by
$$
\omega_t + u\omega_x + v\omega_y = \nu_1 \omega_{yy} + \theta_x
$$
and the mismatch between the vertical dissipation $\omega_{yy}$ and the $x$-derivative $\theta_x$ essentially makes the vertical dissipation useless. In two recent papers in collaboration with Adhikari (\cite{ACW10,ACW11}), we attempted to overcome this difficulty and obtained partial results. This paper completely solves the global regularity problem for (\ref{vbou}). Our major result can be stated as follows.
\begin{thm} \label{major}
Consider the initial-value problem for the anisotropic Boussinesq equations with vertical dissipation (\ref{vbou}). Let $\nu>0$ and $\kappa>0$. Let $(u_0, v_0, \theta_0)\in  H^2(\mathbb{R}^2)$. Then, for any $T>0$, (\ref{vbou}) has a unique classical solution $(u, v, \theta)$ on $[0,T]$ satisfying
$$
(u, v, \theta) \in C([0,T]; H^2(\mathbb{R}^2)).
$$
\end{thm}

In order to prove this theorem, we first discover that the norms of the vertical velocity $v$ in Lebesgue spaces play a crucial role in controlling the Sobolev-norms of the solutions. In fact, it is shown in \cite{ACW11} that
\begin{equation}\label{h2bd}
\|(u, v, \theta)\|_{H^2}^2 + \|\omega^2 + |\nabla\theta|^2\|_{L^2}^2 \le C(\nu,\kappa, T, u_0, v_0, \theta_0)\, \exp\left(\int_0^t \|v(\cdot,\tau)\|^2_{L^\infty}\, d\tau\right).
\end{equation}
We remark that (\ref{h2bd}) involves the estimates of $(u,v,\theta)$ in $L^4$ and ${W^{1,4}}$, which serves as a bridge to the $H^2$-estimate. It does not appear to be plausible to directly show that
$$
\int_0^t \|v(\cdot,\tau)\|^2_{L^\infty}\, d\tau <\infty.
$$
A natural idea is then to estimate $\|v\|_{L^q}$ for $q<\infty$ and obtain upper bounds that are as sharp as possible for large $q$. It was shown in \cite{ACW10} that $\|v\|_{L^q}$ remains finite for all time with an upper bound depending exponentially on $q$. This upper bound was improved to a linear function of $q$ in \cite{ACW11}. We are able to obtain a further improvement in this paper and prove that, for any $2\le q<\infty$,
\begin{equation} \label{qlogq}
\|v(\cdot,t)\|_{L^{q}} \le  B(t) \sqrt{q \log q},
\end{equation}
where $B(t)$ is an explicit integrable function independent of $q$. In order to bound $\|v\|_{L^\infty}$ in terms of $L^q$-bound in \eqref{qlogq}, we prove the following interpolation inequality
\begin{equation} \label{inte}
\|f\|_{L^\infty(\mathbb{R}^2)} \le C \, \sup_{r\ge 2} \frac{\|f\|_{r}}{\sqrt{r \log r }} \,
\left(\log (e + \|f\|_{H^2(\mathbb{R}^2)})\log \log (e + \|f\|_{H^2(\mathbb{R}^2)})\right)^\frac12.
\end{equation}
This delicate inequality together with \eqref{h2bd} and \eqref{qlogq} yields the desired global bound for $\|(u, v, \theta)\|_{H^2}$. The global bound combined with the local existence theory (see, e.g., \cite{CN}) leads to the global regularity result stated in Theorem \ref{major}.

\vskip .1in
Our major effort is devoted to proving the upper bound \eqref{qlogq}. A key ingredient of the proof is the global bounds on the pressure $p$,
$$
\|p(\cdot,t)\|_{L^2} \le C, \qquad \|p(\cdot,t)\|_{L^4} \le C, \qquad \int_0^t \|\nabla p(\cdot, \tau)\|^2_{L^2} \,d\tau \le C,
$$
where $C=C(\nu,\kappa,t, u_0,v_0,\theta_0)$ is a smooth function of $t$ that depends on the parameters $\nu$, $\kappa$ and the initial norm $\|(u_0, v_0, \theta_0)\|_{H^2}$. These bounds for the pressure, in turn, require suitable estimates for $(u, v, \theta)$ in $L^4$ and $L^8$. Another crucial technique is the decomposition of the pressure into low frequency and high frequency parts, which are bounded differently. The separation of the low-high frequencies appears to be necessary in securing a bound of the form in \eqref{qlogq}. The proof of the interpolation inequality \eqref{inte} involves the Littlewood-Paley decomposition and Besov space tools.

\vskip .1in
The rest of the paper is divided into three sections and two appendices. The second section proves the interpolation inequality (\ref{inte}), an inequality for a triple product and several estimates for the low and high frequency parts of a $H^1$-function. The third section establishes several global bounds for the pressure. They rely on the $L^4$ and $L^8$ bounds of the solution. The last section presents the proof of Theorem \ref{major}. The key is the global {\it a priori} bound (\ref{qlogq}), whose detailed proof is also provided in this section. Appendix A contains the description of the Littlewood-Paley decomposition, the Besov spaces, the Triebel-Lizorkin spaces and related facts used in the previous sections. Appendix B provides the technical proof for an inequality presented in the second section.  Throughout the rest of this paper, $\|f\|_{L^r}$ or simply $\|f\|_r$ denotes the norm in the Lebesgue space $L^r$, $\|f\|_{H^s}$ and $\|f\|_{\mathring{H}^s}$ denote the norms in the Sobolev space $H^s$ and the homogeneous Sobolev space $\mathring{H}^s$, respectively.

\vskip .3in
\section{Interpolation inequality and other tools}
\label{tools}

This section presents several inequalities to be used in the subsequent sections. They include the interpolation inequality stated in (\ref{inte}), an inequality for a triple product and suitable bounds for the low and high frequency parts of a $H^1$-function. Some of the proofs involve the Littlewood-Paley decomposition, Besov spaces, Triebel-Lizorkin spaces and related techniques, which are described in Appendix A.

\vskip .1in
\begin{lemma} \label{login}
Let $s>1$ and $f\in H^s(\mathbb{R}^2)$. Assume that
$$
\sup_{r\ge 2} \frac{\|f\|_{r}}{\sqrt{r \log r }} <\infty.
$$
Then there exists a constant $C$ depending on $s$ only such that
\begin{equation} \label{gelog}
\|f\|_{L^\infty(\mathbb{R}^2)} \le C \, \sup_{r\ge 2} \frac{\|f\|_{r}}{\sqrt{r \log r }} \,
\left(\log (e + \|f\|_{H^s(\mathbb{R}^2)})\log \log (e + \|f\|_{H^s(\mathbb{R}^2)})\right)^\frac12.
\end{equation}
\end{lemma}

When $s=2$, (\ref{gelog}) reduces to (\ref{inte}). The proof of this lemma involves the Littlewood-Paley decomposition, Bernstein's inequality and the identification of the inhomogeneous Besov space $B^s_{2,2}$ with $H^s$.

\vskip .1in
\begin{proof}[Proof of Lemma \ref{login}]
By the Littlewood-Paley decomposition, we can write
$$
f = S_{N+1}f+\sum_{j=N+1}^\infty \Delta_j f,
$$
where $\Delta_j$ denotes the Fourier localization operator and
$$
S_{N+1} = \sum_{j=-1}^{N} \Delta_j.
$$
The definitions of $\Delta_j$ and $S_{N}$ are now standard and can be found in several books and many papers (see, e.g., \cite{BL,Che,RS,Tri}). For reader's convenience, they are provided in Appendix A. Therefore,
$$
\|f\|_{\infty} \le \|S_{N+1}f\|_\infty + \sum_{j=N+1}^\infty \|\Delta_j f\|_\infty.
$$
We denote the terms on the right by $I$ and $II$. By Bernstein's inequality (see Appendix A), for any $q\ge 2$,
$$
|I| \le  2^{\frac{2N}{q}} \|S_{N+1}f\|_q \le 2^{\frac{2N}{q}} \|f\|_{q}.
$$
Taking $q=N$, we have
$$
|I| \le 4 \|f\|_{N} \le 4 \sqrt{N \log N}\, \sup_{r\ge 2} \frac{\|f\|_{r}}{\sqrt{r\log r}}.
$$
By Bernstein's inequality again, for any $s>1$,
\begin{eqnarray*}
|II| &\leq & \sum_{j=N+1}^\infty 2^j \|\Delta_jf\|_2 =
\sum_{j=N+1}^\infty 2^{-j(s-1)}\,2^{sj}\|\Delta_jf\|_2\\
&=& C\, 2^{-(N+1)(s-1)}\, \|f\|_{B^s_{2,2}}.
\end{eqnarray*}
where $C$ is a constant depending on $s$ only.  By identifying $B^s_{2,2}$ with $H^s$, we obtain
\begin{eqnarray*}
\|f\|_{\infty}\leq 4 \sqrt{N \log N}\,
\sup_{r\ge 2} \frac{\|f\|_{r}}{\sqrt{r\log r}}  +
C\,2^{-(N+1)(s-1)}\, \|f\|_{H^s}.
\end{eqnarray*}
We obtain the desired inequality \eqref{gelog} by taking
$$
N =\left[\frac{1}{s-1} \log_2(e+\|f\|_{H^s})\right],
$$
where $\left[a\right]$ denotes the largest integer less than or equal to $a$.
\end{proof}

\vskip .1in
The next lemma bounds the triple product in terms of the Lebesgue norms of the functions and their directional derivatives.
\begin{lemma} \label{triple}
Let $q\in [2,\infty)$. Assume that $f, g, g_y, h_x\in L^2(\mathbb{R}^2)$ and $h\in L^{2(q-1)}(\mathbb{R}^2)$. Then
\begin{equation} \label{qtri}
 \iint_{\mathbb{R}^2}| f \, g\, h|  \;dxdy  \le C \, \|f\|_2 \, \|g\|_{2}^{1-\frac1q} \|g_y\|_2^{\frac1q} \,\|h\|_{{2(q-1)}}^{1-\frac1q} \|h_x\|_2^{\frac1q}.
\end{equation}
where $C$ is a constant depending on $q$ only. Two special cases of (\ref{qtri}) are
\begin{equation} \label{1/3}
\iint | f \, g\, h|  \;dxdy
\leq C\; \|f\|_2 \; \|g\|_2^{\frac23}   \; \|g_y\|_2^{\frac13}   \; \|h\|_4^{\frac23}
\; \| h_{x}\|_2^{\frac13}
\end{equation}
and
\begin{equation} \label{1/2}
\iint | f \, g\, h|  \;dxdy
\leq C\; \|f\|_2 \; \|g\|_2^{\frac12}   \; \|g_y\|_2^{\frac12}   \; \|h\|_2^{\frac12}
\; \| h_{x}\|_2^{\frac12}.
\end{equation}
\end{lemma}

This lemma generalizes an inequality in \cite{CaoWu}. It is clear from (\ref{qtri}) that, as $q$ increases, $\|h\|_{{2(q-1)}}$ absorbs a higher power while the power of $\|h_x\|_{2}$ decreases. This generalization allows us to select suitable $q$'s to obtain our desired estimates. The two particular special cases in \eqref{1/3} and \eqref{1/2} will be very useful in the proof of Theorem \ref{major} in Section \ref{majorproof}. The proof of Lemma \ref{triple} will be provided in Appendix B.

\vskip .1in
The next two lemmas bound the norms of the low and high frequency parts of a $H^1$ function in Lebesgue spaces. We recall the definition of the Fourier and the inverse Fourier transforms:
\begin{eqnarray*}
\mathcal{F}f (\xi) = \widehat{f}(\xi) = (2\pi)^{-\frac{d}{2}} \int_{\mathbb{R}^d} e^{-ix\cdot \xi} f(x)\,dx, \\
\mathcal{F}^{-1} f(x) =\breve{f}(x)  = (2\pi)^{-\frac{d}{2}} \int_{\mathbb{R}^d} e^{ix\cdot \xi} f(\xi)\,d\xi.
\end{eqnarray*}
\begin{lemma} \label{ha}
Let $f\in H^1(\mathbb{R}^2)$. Let $R>0$. Denote by $B(0,R)$ the ball centered at zero with radius $R$ and by $\chi_{B(0,R)}$ the characteristic function on $B(0,R)$. Write
\begin{equation}\label{dec}
f = \overline{f} + \widetilde{f}\quad\mbox{with}\quad \overline{f}= \mathcal{F}^{-1}(\chi_{B(0,R)} \mathcal{F}f) \quad\mbox{and}\quad \widetilde{f}= \mathcal{F}^{-1}((1-\chi_{B(0,R)}) \mathcal{F}f).
\end{equation}
Then we have the following estimates for $\overline{f}$ and $\widetilde{f}$.
\begin{enumerate}
\item[(1)] There exists a pure constant $C$ independent of $f$ and $R$ such that
\begin{equation}\label{goo}
\|\overline{f}\|_{L^\infty(\mathbb{R}^2)} \le C\, \sqrt{\log R} \, \|f\|_{H^1(\mathbb{R}^2)}.
\end{equation}
\item[(2)] For any $2\le q<\infty$, there is a constant independent of $q$, $R$ and $f$ such that
\begin{equation}\label{good}
\|\widetilde{f}\|_{L^{q}(\mathbb{R}^2)} \le C\,\frac{q}{R^{\frac2{q}}} \, \|\widetilde{f}\|_{H^1(\mathbb{R}^2)} \le C\,\frac{q}{R^{\frac2{q}}} \, \|f\|_{H^1(\mathbb{R}^2)}
\end{equation}
In particular, for $q=4$,
$$
\|\widetilde{f}\|_{L^4(\mathbb{R}^2)} \le \frac{C}{\sqrt{R}} \, \|f\|_{H^1(\mathbb{R}^2)}.
$$
\end{enumerate}
\end{lemma}

\begin{proof}The proof of (\ref{goo}) is very easy.
\begin{eqnarray*}
\|\overline{f}\|_{L^\infty} \le \|\widehat{\overline{f}}\|_{L^1}
&=& \int_{|\xi|\le R} |\widehat{\overline{f}}(\xi)|\,d\xi\\
&=& \int_{|\xi|\le R} (1+ |\xi|^2)^{-\frac12} (1+ |\xi|^2)^{\frac12} |\widehat{\overline{f}}(\xi)| \,d\xi \\
&\le& C\, \sqrt{\log R} \, \|f\|_{H^1}.
\end{eqnarray*}
To prove (\ref{good}), we first recall the embedding relations: for any $1\le q <\infty$,
$$
\mathring{B}^0_{q, \min\{q,2\}} \hookrightarrow L^q \hookrightarrow \mathring{B}^0_{q, \max\{q,2\}},
$$
where $\mathring{B}^s_{q,r}$ denotes the homogenous Besov space (see Appendix A).  In particular, for $2\le q <\infty$,
$$
\|\widetilde{f}\|_{L^q} \le C\,q\, \|\widetilde{f}\|_{\mathring{B}^0_{q,2}} =C\,q\, \left[\sum_{j=-\infty}^\infty \|\Delta_j \widetilde{f}\|^2_{L^q}\right]^{\frac12},
$$
where $C$ is a constant independent of $q$.
By definition,
$$
\widehat{\Delta_j \widetilde{f}}(\xi) = \widehat{\Phi}_j (\xi) \widehat{\widetilde{f}}(\xi), \qquad \mbox{supp} \widehat{\Phi}_j \subset \{\xi: \, 2^{j-1} <|\xi| \le 2^{j+1}\}.
$$
Since supp$\widehat{\widetilde{f}} \subset \{\xi: |\xi|\ge R\}$, $\Delta_j \widetilde{f} \equiv 0$
for $j\le j_0 \equiv [\log_2 R] - 1$. By Bernstein's inequality,
\begin{eqnarray*}
\|\widetilde{f}\|_{L^q} &\le&  C\,q\,\left[\sum_{j=j_0}^{\infty}  \|\Delta_j \widetilde{f}\|^2_{L^q}\right]^\frac12  \le C\,q\,\left[\sum_{j=j_0}^{\infty} 2^{4j(\frac12-\frac1q)} \|\Delta_j \widetilde{f}\|^2_{L^2}\right]^\frac12
\\
&\le& C\,q\, 2^{-\frac{2j_0}{q}}\,\|\widetilde{f}\|_{H^1} \le \frac{C\,q}{R^{\frac2{q}}} \, \|\widetilde{f}\|_{H^1}.
\end{eqnarray*}
This completes the proof of Lemma \ref{ha}.
\end{proof}

\begin{lemma} \label{lpt}
Let $1<q<\infty$. Let $f\in L^q(\mathbb{R}^d)$ and let $\widetilde{f}$ be defined as in (\ref{dec}). Then, for a constant $C$ depending on $q$ only such that
$$
\|\widetilde{f}\|_{L^q(\mathbb{R}^d)} \le C\, \|f\|_{L^q(\mathbb{R}^d)}.
$$
\end{lemma}

\begin{proof}
For any $1<q<\infty$, we have the equivalence relation  $L^q \sim F^0_{q,2}$, where
$F^s_{q,r}$ denotes the Triebel-Lizorkin space containing tempered distributions $f$ such that
$$
\|f\|_{F^s_{q,r}} \equiv \|2^{sj} |\Delta_j f|\|_{L^q(l^r)} =\left\| \left[\sum_{j={-1}}^\infty 2^{sjr} |\Delta_j f|^r \right]^{\frac1r}  \right\|_{L^q} <\infty.
$$
More information on Triebel-Lizorkin spaces can be found in Appendix A and some books (see, e.g., \cite{RS,Tri}). Therefore, for two constants $C_1$ and $C_2$ depending on $q$ only,
\begin{equation}\label{jjj}
C_1 \|f\|_{F^0_{p,2}} \le \|f\|_{L^p} \le C_2 \|f\|_{F^0_{p,2}}.
\end{equation}
By the definition of $\widetilde{f}$, $\Delta_j \widetilde{f} =0$ for $j\le j_0 = [\log R]-1$.  Applying (\ref{jjj}), we have
\begin{eqnarray*}
\|\widetilde{f}\|_{L^p} &\le& C_2 \|\widetilde{f}\|_{F^0_{p,2}} =C_2\, \left\| \left[\sum_{j={-1}}^\infty 2^{2sj} |\Delta_j \widetilde{f}|^2 \right]^{\frac12}  \right\|_{L^p}
\le  C_2\,  \left\| \left[\sum_{j={j_0}}^\infty 2^{2sj} |\Delta_j f|^2 \right]^{\frac12}  \right\|_{L^p} \\
&\le& C_2\,  \left\| \left[\sum_{j={-1}}^\infty 2^{2sj} |\Delta_j f|^2 \right]^{\frac12}  \right\|_{L^p} \le C_1 C_2\|f\|_{L^p}.
\end{eqnarray*}
This completes the proof of Lemma \ref{lpt}.
\end{proof}

\vskip .3in
\section{Global bounds for the pressure}

This section establishes several global {\it a priori} bounds for the pressure.
These bounds are crucial in the proof of the global bound \eqref{qlogq},
which, in turn, is one of the major components in the proof of
Theorem \ref{major}. These bounds for the pressure are simultaneously obtained with the $L^4$ and $L^8$ bounds of $(u,v,\theta)$. We state these bounds in two propositions.

\begin{prop} \label{ppr1}
Let $(u_0, v_0, \theta_0)\in H^2(\mathbb{R}^2)$ and let $(u, v, \theta)$ be the corresponding classical solution of (\ref{vbou}). Then
\begin{equation}\label{l4uv}
\| (u(t), v(t))\|_4^4 +  \nu\, \int_0^t \||(u_y(\tau), v_y(\tau))|\, |(u(\tau), v(\tau))| \|_2^2 \,d\tau \leq M_1(t),
\end{equation}
\begin{equation}\label{pineq}
\|p(\cdot,t)\|_2 \le M_2(t), \qquad \int_0^t \|\nabla p(\cdot, \tau)\|^2_2 \,d\tau \le M_3(t),
\end{equation}
where $M_1, M_2$ and $M_3$ are explicit smooth functions of $t\in[0,\infty)$ that depend on $\nu, \kappa$ and the initial norm $\|(u_0, v_0, \theta_0)\|_{H^2}$.
\end{prop}

The estimates in Proposition \ref{ppr1} have been partially obtained in \cite{ACW11}. To be self-contained, we include these estimates with a simplified proof.

\begin{prop} \label{ppr2}
Let $(u_0, v_0, \theta_0)\in H^2(\mathbb{R}^2)$ and let $(u, v, \theta)$ be the corresponding classical solution of (\ref{vbou}). Then
\begin{equation}\label{v8}
\|v(t)\|_8  \leq M_4(t),
\end{equation}
\begin{equation}\label{uy2}
\|u_y(t)\|_2^2 +  \nu\, \int_0^t \|u_{yy}(\tau)\|_2^2 \,d\tau \le  M_5(t),
\end{equation}
\begin{equation}\label{p4}
\|p(t)\|_4 \le M_6(t),
\end{equation}
where $M_4, M_5$ and $M_6$ are explicit smooth functions of $t\in[0,\infty)$ that depend on $\nu, \kappa$ and the initial norm $\|(u_0, v_0, \theta_0)\|_{H^2}$.
\end{prop}

\vskip .1in
To prove the propositions, we first recall  the following lemma.
\begin{lemma} \label{basicp}
Let $(u_0, v_0, \theta_0)\in H^2(\mathbb{R}^2)$ and let $(u, v, \theta)$ be the corresponding classical solution of (\ref{vbou}). Then
\begin{equation}\label{l2uv}
\|(u(t), v(t))\|_2^2 + 2 \nu\, \int_0^t \|(u_y(\tau), v_y(\tau))\|_2^2 \,d\tau =\left(\|(u_0, v_0)\|_2 + t\,\|\theta_0\|_2\right)^2
\end{equation}
and, for any $q\ge 2$,
\begin{equation}\label{l2tt}
\|\theta(t)\|_q^q + \kappa \,q\,(q-1)\, \int_0^t \|\theta_y\,|\theta|^{\frac{q-2}{2}}(\tau)\|_2^2 \,d\tau = \|\theta_0\|_q^q.
\end{equation}
In particular, for $2 \le q\le \infty$,
\begin{equation}\label{thq}
\|\theta(t)\|_q\le \|\theta_0\|_q.
\end{equation}
\end{lemma}

\vskip .1in
\begin{proof}[Proof of Proposition \ref{ppr1}]
Taking the inner product of the first equation in (\ref{vbou}) with $u\,(|u|^2+ |v|^{2})$ and
the second equation in (\ref{vbou}) with $v\,(|u|^2+ |v|^{2})$ and integrating by parts, we obtain
\begin{eqnarray*}
&&  \frac{1}{4}\frac{d}{dt} \int (|u|^2+ |v|^{2})^2 \, + \, \nu  \int (|u|^2+ |v|^{2}) (|u_y|^2 + |v_y|^{2})
+ 2 \nu  \int  (u\, u_y + v\, v_y)^{2}   \\
&& \qquad
=-\int p_x \, u\,(|u|^2+ |v|^{2}) +  \,\int p \, v_y \, (|u|^2+ |v|^{2}) +
\,2 \int p \, v \, (u\, u_y+ v\, v_y)\\
&&  \qquad \quad
+  \int \theta \, v \,(|u|^2+ |v|^{2}).
\end{eqnarray*}
The terms on the right can be estimated as follows. By H\"{o}der's inequality,
\begin{eqnarray*}
\left| \int  \theta \, v \,(|u|^2+ |v|^{2}) \right| \le C \|\theta\|_{4}\,\||u|^2+ |v|^{2} \|_{2}^{3/2}.\label{eas-4}
\end{eqnarray*}
By Lemma \ref{triple},
\begin{eqnarray*}
&& \hskip -.3in
\left|  \int p_x \, u \,\,(|u|^2+ |v|^{2})  \right|\le C \|p_x\|_{2}\, \|u\|_2^{1/2} \|u_x\|_2^{1/2}
\left\||u|^2+ |v|^{2} \right\|_{2}^{1/2} \left\|u u_y+ vv_y \right\|_{2}^{1/2}, \label{pterm-4}  \\
&&  \hskip -.3in
\left|  \int p \, v_y \, (|u|^2+ |v|^{2}) \right| \le C \|p\|_2^{1/2} \|p_x\|_2^{1/2} \,
\|v_y\|_2 \,\||u|^2+ |v|^{2} \|_{2}^{1/2}\left\|u u_y+ vv_y \right\|_{2}^{1/2},  \\
&&  \hskip -.3in
2 \left| \int p \, v \, (u\, u_y+ v\, v_y) \right| \le C \|p\|_2^{1/2} \|p_x\|_2^{1/2} \, \|v\|_2^{1/2} \|v_y\|_2^{1/2}
\left\|u u_y+ vv_y \right\|_{2}. \label{pterm-44}
\end{eqnarray*}
To further the estimate for $p$, we take the divergence of the first two equations in (\ref{vbou}) to get
\begin{eqnarray*}
\Delta p &=& -(uu_x +vu_y)_x - (uv_x +vv_y)_y + \theta_y \\
&=& -2 (v u_y)_x  - 2 (v\,v_y)_y + \theta_y.
\end{eqnarray*}
By standard bounds for the Riesz transforms, we have
\begin{eqnarray}
 \|p\|_2 & \le & C \left( \left\||u|^2+ |v|^{2} \right\|_{2} + \|\theta\|_{2} \right), \label{nap-4-1}  \\
 \|\nabla p\|_{2} & \le & C \left( \left\|v\,u_y \right\|_{2}
+ \left\|v\,v_y \right\|_{2}  + \|\theta\|_{2} \right). \label{nap-4-2}
\end{eqnarray}
Thus, we obtain
\begin{eqnarray*} \label{ebase-44}
&&  \frac{1}{4}\frac{d}{dt} \int (|u|^2+ |v|^{2})^2 \, + \, \nu  \int (|u|^2+ |v|^{2}) (|u_y|^2 + |v_y|^{2})
+ 2 \nu  \int  (u\, u_y + v\, v_y)^{2}   \\
&& \qquad
\leq C  \left( \left\|v\,u_y \right\|_{2}
+ \left\|v\,v_y \right\|_{2}  + \|\theta\|_{2} \right) \; \|u\|_{2}^{1/2} \|v_y\|_2^{1/2} \;  \|u^2+ v^2 \|_{2}^{1/2}
\left\|u u_y+ vv_y \right\|_{2}^{1/2} \\
&& \qquad \quad
+ \,C \left(\|u^2+ v^2 \|_{2} + \|\theta\|_{2} \right)^{1/2}   \left( \left\|v\,u_y \right\|_{2}
+ \left\|v\,v_y \right\|_{2}  + \|\theta\|_{2} \right)^{1/2} \,
\|v_y\|_2 \\
&& \qquad \qquad \times \||u|^2+ |v|^{2} \|_{2}^{1/2}\left\|u u_y+ vv_y \right\|_{2}^{1/2}\\
&& \qquad \quad
+ \,C \left(\|u^2+ v^2 \|_{2} + \|\theta\|_{2} \right)^{1/2}   \left( \left\|v\,u_y
\right\|_{2}
+ \left\|v\,v_y \right\|_{2}  + \|\theta\|_{2} \right)^{1/2}  \, \|v\|_2^{1/2} \\
&& \qquad \qquad \times
\|v_y\|_2^{1/2} \left\|u u_y+ vv_y \right\|_{2} +\, C \|\theta\|_{4}\,\||u|^2+ |v|^{2} \|_{2}^{3/2}.
\end{eqnarray*}
By Young's inequality,
\begin{eqnarray*} \label{ebase-4-1}
&&  \frac{d}{dt} \int (|u|^2+ |v|^{2})^2 \, + \, \nu  \int (|u|^2+ |v|^{2}) (|u_y|^2 + |v_y|^{2}) +  \nu  \int  (u\, u_y + v\, v_y)^{2}   \\
&&  \qquad
\leq C(1+\|u\|_2^2+\|v\|_2^2)(1+\; \|u\|_{H^1}^2) \;  \|(u^2+ v)^{2} \|_{2}^2 + C \|\theta\|_{2}^2 \; \|u\|_{H^1}^2\\
&&   \qquad \quad
+ \,C \|\theta\|_{4}^4 +C \|\theta\|_{2}^4.
\end{eqnarray*}
By Gronwall inequality and applying Lemma \ref{basicp}, we reach
\begin{eqnarray*}
&& \int (|u|^2+ |v|^{2})^2 \, + \, \nu  \int_0^t \int (|u|^2+ |v|^{2}) (|u_y|^2 + |v_y|^{2})
+  \nu  \int_0^t \int  (u\, u_y + v\, v_y)^{2}   \\
&& \qquad
\leq e^{C\,\left(t+ \int_0^t \|u\|_{H^1}^2 d\tau\right)} \; \left( \||u_0|^2+ |v_0|^{2} \|_{2}^2 + C \|\theta_0\|_{2}^2 \; \int_0^t \|u\|_{H^1}^2\,d\tau
+ C \|\theta_0\|_{4}^4 \; t \right).
\end{eqnarray*}
Using the fact that $u_x =-v_y$ and Lemma \ref{basicp}, we see that $\|u\|_{H^1}^2$ is time integrable and the inequality above verifies \eqref{l4uv}.
In particular, by \eqref{nap-4-1} and \eqref{nap-4-2}, we obtain \eqref{pineq}. This completes thee proof of Proposition \ref{ppr1}.
\end{proof}

\vskip .1in
\begin{proof}[Proof of Proposition \ref{ppr2}]
Taking the inner product of the second equation in (\ref{vbou}) with $v\,|v|^6$ and integrating by parts, we obtain
\begin{eqnarray*}
&&  \frac{1}{8}\frac{d}{dt}\int |v|^8 \, + 7 \nu  \int |v|^{6}\, |v_y|^{2} =   7\,\int p \, v_y \, |v|^{6} +  \int \theta v |v|^{6}.
\end{eqnarray*}
The terms on the right can be bound by
\begin{eqnarray*}
 7 \int p \, |v_y| \, |v|^{6} +  \int \theta |v|^{7} &\leq&  C \|p\|_8 \|v\|_{8}^{3} \|v^3 v_y\|_2 +\|\theta\|_8  \|v\|_8^7   \\
&\leq& C \|p\|^{1/4}\|p\|_{H^1}^{3/4} \|v\|_{8}^{3} \|v^3 v_y\|_2 +\|\theta\|_8  \|v\|_8^7.
\end{eqnarray*}
By Young's inequality, we get
$$
 \frac{d}{dt} \|v\|_8^8 \,  + \, \nu  \int |v|^{6}\, |v_y|^{2}
\leq C    \|p\|^{1/2}\|p\|_{H^1}^{3/2} \|v\|_{8}^{6} +\|\theta\|_8  \|v\|_8^7.
$$
Thus,
\begin{eqnarray*}
&&  \frac{d}{dt} \|v\|_8^2 \,
\leq     C \|p\|^{1/2}\|p\|_{H^1}^{3/2}  +\|\theta\|_8^2+   \|v\|_8^2.
\end{eqnarray*}
Applying Gronwall inequality then yields \eqref{v8}. We now prove (\ref{uy2}). Taking the inner product of the first equation in (\ref{vbou}) with $-u_{yy}$  and integrating by parts, we obtain
$$
 \frac{1}{2}\frac{d}{dt} \|u_y\|_2^2 \, + \, \nu  \|u_{yy}\|_2^2
=-\int p_x \, u_{yy}$$
Applying H\"{o}lder's and Young's inequalities, we get
\begin{eqnarray*}
&&  \frac{d}{dt} \|u_y\|_2^2 \, + \, \nu  \|u_{yy}\|_2^2
  \le 2 \|p_x\|_2^2.
\end{eqnarray*}
Therefore,
\begin{eqnarray} \label{ebase-4-2}
&& \|u_y\|_2^2 \, + \, \nu  \int_0^t\|u_{yy}\|_2^2 d\tau \le C \int_0^t \|p\|_{H^1}^2 d\tau +  \|u_{0y}\|_2^2.
\end{eqnarray}
To prove (\ref{p4}), we recall that
\begin{eqnarray*}
\Delta p &=& -2 (v u_y)_x  -  (v^2)_{yy} + \theta_y.
\end{eqnarray*}
By the Hardy-Littlewood-Sobolev inequality and standard bounds for the Riesz transforms, we have
\begin{eqnarray*}
\|p(t)\|_4  &\le&  C \left( \|(-\Delta)^{-1} (v u_y)_x\|_4 + \|(-\Delta)^{-1} (v^2)_{yy}\|_4 + \|(-\Delta)^{-1} \theta_y \|_4\right)  \\
&\le& C \left( \|v u_y \|_{4/3} +  \|v\|_8^2  + \|\theta\|_{4/3} \right)  \\
&\le& C \left( \| v\|_4\, \| u_y \|_{2} +  \|v\|_8^2  + \|\theta_0\|_{4/3} \right).
\end{eqnarray*}
Therefore, (\ref{p4}) holds. This completes the proof of Proposition \ref{ppr2}.
\end{proof}

\vskip .3in
\section{Proof of Theorem \ref{major}}
\label{majorproof}

This section is devoted to proving Theorem \ref{major}. A major component of the proof is the global bound (\ref{qlogq}). A precise statement of this inequality is contained in the following proposition.
\begin{prop} \label{logeq}
Let $(u_0, v_0, \theta_0)\in H^2(\mathbb{R}^2)$ and let $(u, v, \theta)$ be the corresponding classical solution of (\ref{vbou}). Then,
\begin{equation}
\sup_{r\ge 2} \frac{\|v(t) \|_{L^{2r}}}{\sqrt{r \log r}} \le \sup_{r\ge 2} \frac{\|v_0 \|_{L^{2r}}}{\sqrt{r \log r}} + B(t),
\end{equation}
where $B(t)$ is an explicit integrable function of $t\in [0,\infty)$ that depends on $\nu, \kappa$ and the initial norm $\|(u_0, v_0, \theta_0)\|_{H^2}$.
\end{prop}

\vskip .1in
With this global bound at our disposal, we are ready to prove Theorem \ref{major}.
\begin{proof}[Proof of Theorem \ref{major}]
It follows from Proposition 5.3 of \cite{ACW11} that the quantity
$$
Y(t) = \|\omega\|_{H^1}^2 + \|\theta\|_{H^2}^2 + \|\omega^2+|\nabla \theta|^2\|_2^2
$$
satisfies
\begin{eqnarray*}
&&  \frac{d}{dt} Y(t) + \|\omega_y\|_{H^1}^2 + \|\theta_y\|_{H^2}^2 + \int (\omega^2+|\nabla \theta|^2)\,\left(\omega_y^2+|\nabla \theta_y|^2\right)
+ \int (\omega \omega_y+\nabla \theta \cdot \nabla \theta_y)^2\\
&& \qquad \le C\left(1+\|\theta_0\|_\infty^2+ \|v\|_\infty^2 +\|u_y\|_2^2 + (1+\|u\|_2^2)\|v_y\|_2^2\right)\,Y(t).
\end{eqnarray*}
By Proposition \ref{login},
$$
\|v\|_{\infty} \le C \, \sup_{r\ge 2} \frac{\|v\|_{r}}{\sqrt{r \log r }} \,
\left(\log (e + \|v\|_{H^2})\log \log (e + \|v\|_{H^2})\right)^\frac12.
$$
Applying Proposition \ref{logeq} and using the simple fact that $\|v\|^2_{H^2} \le \|\omega\|^2_{H^1} \le Y(t)$, we obtain
$$
 \frac{d}{dt} Y(t)  \le A(t) Y(t) + C\, B^2(t)\, Y(t)\, \log (e + Y(t))\,\log\log (e + Y(t)),
$$
where $A(t)=C\, \left(1+\|\theta_0\|_\infty^2+ \|u_y\|_2^2 + (1+\|u\|_2^2)\|v_y\|_2^2\right)$. An application of Gronwall's inequality then concludes the proof of Theorem \ref{major}.
\end{proof}

\vskip .1in
Finally we prove Proposition \ref{logeq}.
\begin{proof}[Proof of Proposition \ref{logeq}]
 Taking the inner product of the second equation in (\ref{vbou}) with $v\,|v|^{2r-2}$ and integrating by parts, we obtain
\begin{eqnarray}
&& \label{ebase}
\frac{1}{2r}\frac{d}{dt} \int |v|^{2r} \, + \, \nu (2r-1) \int v_y^2 \,|v|^{2r-2}  \nonumber \\
&& \qquad
= (2r-1)\,\int p \, v_y \,\,|v|^{2r-2} +  \int \theta \, v \,|v|^{2r-2} \nonumber \\
&& \qquad
= (2r-1)\,\int \overline{p} \, v_y \,\,|v|^{2r-2} + (2r-1)\,\int \tilde{p} \, v_y \,\,|v|^{2r-2} +  \int \theta \, v \,|v|^{2r-2}. \label{root}
\end{eqnarray}
By H\"{o}lder's inequality,
\begin{eqnarray}
&&  \int \theta \, v \,|v|^{2r-2} \le \|\theta\|_{2r}\,\|v\|_{2r}^{2r-1}, \label{eas}\\
&& \int \overline{p}  \, v_y \,\,|v|^{2r-2}  \le \|\overline{p}\|_{\infty}\,  \|v^{r-1}\|_{2} \, \|v_y v^{r-1}\|_2. \nonumber
\end{eqnarray}
Applying Lemma \ref{triple}, we have
$$
\int \tilde{p} \, v_y \,\,|v|^{2r-2} \le C \|\tilde{p}\|_4^{\frac23}\, \|\tilde{p}_x\|_2^{\frac13}\, \|v^{r-1}\|_2^{\frac23}\, \|(r-1) v_y v^{r-2}\|_2^{\frac13} \, \|v_y v^{r-1}\|_2.
$$
Furthermore, by H\"{o}lder's inequality,
\begin{eqnarray*}
&&
\left\||v|^{r-1} \right\|_2  =  \left\|v \right\|_{2(r-1)}^{r-1}
\leq   \|v\|_2^{\frac{1}{r-1}}
\left\|v \right\|_{2r}^{\frac{r(r-2)}{r-1}},  \\
&&
\left\||v|^{r-2}v_y \right\|_2^2 =  \int  |v|^{2(r-2)}v_y^2   = \int  |v|^{2(r-2)}v_y^{\frac{2(r-2)}{r-1}}
v_y^{\frac{2}{r-1}}   \le
 \|v_y\|_2^{\frac{2}{r-1}} \left\|v_y \,|v|^{r-1}\right\|_2^{\frac{2(r-2)}{r-1}}.
\end{eqnarray*}
Therefore,
\begin{eqnarray*}
\int \overline{p}  \, v_y \,\,|v|^{2r-2}  &\le&  C \,\|\overline{p}\|_{\infty}\, \|v\|_2^{\frac{1}{r-1}} \left\|v \right\|_{2r}^{\frac{r(r-2)}{r-1}}\,\|v_y v^{r-1}\|_2, \\
\int \tilde{p} \, v_y \,\,|v|^{2r-2} &\le& C\, (r-1)^{\frac13} \,\|\tilde{p}\|_4^{\frac23}\, \|\tilde{p}_x\|_2^{\frac13}\,\|v\|_{2}^{\frac{2}{3(r-1)}} \left\|v \right\|_{2r}^{\frac{2r(r-2)}{3(r-1)}} \nonumber\\
&&\quad \times \|v_y\|_2^{\frac{1}{3(r-1)}} \left\|v_y \,|v|^{r-1}\right\|_2^{1+ \frac{(r-2)}{3(r-1)}}.
\end{eqnarray*}
By Young's inequality and Lemma \ref{ha},
\begin{eqnarray}
(2r-1) \int \overline{p}  \, v_y \,\,|v|^{2r-2} &\le& \frac{\nu}{4} (2r-1) \|v_y v^{r-1}\|_2^2 \nonumber\\
&& + \, C (2r-1) (\log R)\, \|p\|^2_{H^1} \|v\|_2^{\frac{2}{r-1}} \,\|v\|_{2r}^{2r-2-\frac{2}{r-1}}. \label{pbe}
\end{eqnarray}
By Young's inequality and Lemmas \ref{ha} and \ref{lpt},
\begin{eqnarray}
(2r-1)\int \tilde{p} \, v_y \,\,|v|^{2r-2} &\le& \frac{\nu}{4} (2r-1) \|v_y v^{r-1}\|_2^2 + C\, (2r-1)(r-1)^{\frac{2r-2}{2r-1}}  \nonumber\\ && \times \|\tilde{p}\|_4^{\frac{4(r-1)}{2r-1}} \|\tilde{p}_x\|_2^{\frac{2(r-1)}{2r-1}} \|v_y\|_2^{\frac{2}{2r-1}}
 \|v\|_{2}^{\frac{4}{2r-1}} \left\|v \right\|_{2r}^{2r-2-\frac{2(r+1)}{2r-1}}
\nonumber\\
&\le& \frac{\nu}{4} (2r-1) \|v_y v^{r-1}\|_2^2 + C\, (2r-1)(r-1)^{\frac{2r-2}{2r-1}} R^{-\frac{r-1}{2r-1}} \nonumber\\
&& \times \|p\|_{L^4}^{\frac{2r-2}{2r-1}} \|p\|_{H^1}^{\frac{4r-4}{2r-1}} \|v_y\|_2^{\frac{2}{2r-1}} \|v\|_{2}^{\frac{4}{2r-1}} \left\|v \right\|_{2r}^{2r-3-\frac{3}{2r-1}}.  \label{pte}
\end{eqnarray}
Without loss of generality, we assume $\|v\|_{2r} \ge 1$. Inserting (\ref{eas}),(\ref{pbe}) and (\ref{pte}) in (\ref{root}), we have
\begin{eqnarray*}
&&
\frac1{2r} \frac{d}{dt} \|v\|_{L^{2r}}^{2r} + \frac\nu{2} (2r-1) \int v_y^2 |v|^{2r-2} \,dx \\
&& \qquad \le C (2r-1) (\log R)\, \|p\|^2_{H^1} \|v\|_2^{\frac{2}{r-1}} \,\|v\|_{2r}^{2r-2}\\
&& \qquad
+C\, (2r-1)(r-1)^{\frac{2r-2}{2r-1}} R^{-\frac{r-1}{2r-1}} \|p\|_{L^4}^{\frac{2r-2}{2r-1}} \|p\|_{H^1}^{\frac{4r-4}{2r-1}} \|v_y\|_2^{\frac{2}{2r-1}} \|v\|_{2}^{\frac{4}{2r-1}} \left\|v \right\|_{2r}^{2r-2}\\
&&\qquad + \|\theta\|_{L^{2r}} \, \|v\|_{L^{2r}}^{2r-1}.
\end{eqnarray*}
Especially,
\begin{eqnarray*}
\frac{d}{dt} \|v\|_{L^{2r}}^2 &\le&  C (2r-1) (\log R)\, \|p\|^2_{H^1} \|v\|_2^{\frac{2}{r-1}}
\\
&& + \,C\, (2r-1)(r-1)^{\frac{2r-2}{2r-1}} R^{-\frac{r-1}{2r-1}} \|p\|_{L^4}^{\frac{2r-2}{2r-1}} \left(\|p\|_{H^1}^2 + \|v_y\|^2_2\right) \|v\|_{2}^{\frac{4}{2r-1}}\\
&&  + \, \|\theta\|_{L^{2r}}^2 + \|v\|_{L^{2r}}^2.
\end{eqnarray*}
Taking $R= (2r-1)^\frac{2r-1}{2r-2} (r-1)^2$, integrating in time and applying Propositions \ref{ppr1} and \ref{ppr2}, we obtain
$$
\|v(t) \|_{L^{2r}}^2 \le \|v_0 \|_{L^{2r}}^2 + B_1(t) r \log r + B_2(t),
$$
where $B_1$ and $B_2$ are explicit integrable functions. Therefore,
$$
\sup_{r\ge 2} \frac{\|v(t) \|_{L^{2r}}^2}{r \log r} \le \sup_{r\ge 2} \frac{\|v_0 \|_{L^{2r}}^2}{r \log r} + (B_1(t) + B_2(t)).
$$
This completes the proof of Proposition \ref{logeq}.
\end{proof}

\vskip .4in
\appendix

\section{Besov and Triebel-Lizorkin spaces}
\label{Besov}

This appendix provides the definitions of Besov and Triebel-Lizorkin spaces and some related facts used in the previous sections. Materials presented in this appendix can be found in several books and papers (see, e.g., \cite{BL,Che,RS,Tri}).

\vskip .1in
\vspace{.08in} We start with several notation. $\mathcal{S}$ denotes
the usual Schwarz class and ${\mathcal S}'$ its dual, the space of
tempered distributions. ${\mathcal S}_0$ denotes a subspace of ${\mathcal
S}$ defined by
$$
{\mathcal S}_0 = \left\{ \phi\in {\mathcal S}: \,\, \int_{\mathbb{R}^d}
\phi(x)\, x^\gamma \,dx =0, \,|\gamma| =0,1,2,\cdots \right\}
$$
and ${\mathcal S}_0'$ denotes its dual. ${\mathcal S}_0'$ can be identified
as
$$
{\mathcal S}_0' = {\mathcal S}' / {\mathcal S}_0^\perp = {\mathcal S}' /{\mathcal P}
$$
where ${\mathcal P}$ denotes the space of multinomials.

\vspace{.1in} To introduce the Littlewood-Paley decomposition, we
write for each $j\in \mathbb{Z}$
\begin{equation}\label{aj}
A_j =\left\{ \xi \in {\Bbb R}^d: \,\, 2^{j-1} \le |\xi| <
2^{j+1}\right\}.
\end{equation}
The Littlewood-Paley decomposition asserts the existence of a
sequence of functions $\{\Phi_j\}_{j\in {\Bbb Z}}\in {\mathcal S}$ such
that
$$
\mbox{supp} \widehat{\Phi}_j \subset A_j, \qquad
\widehat{\Phi}_j(\xi) = \widehat{\Phi}_0(2^{-j} \xi)
\quad\mbox{or}\quad \Phi_j (x) =2^{jd} \Phi_0(2^j x),
$$
and
$$
\sum_{j=-\infty}^\infty \widehat{\Phi}_j(\xi) = \left\{
\begin{array}{ll}
1&,\quad \mbox{if}\,\,\xi\in {\Bbb R}^d\setminus \{0\},\\
0&,\quad \mbox{if}\,\,\xi=0.
\end{array}
\right.
$$
Therefore, for a general function $\psi\in {\mathcal S}$, we have
$$
\sum_{j=-\infty}^\infty \widehat{\Phi}_j(\xi)
\widehat{\psi}(\xi)=\widehat{\psi}(\xi) \quad\mbox{for $\xi\in {\Bbb
R}^d\setminus \{0\}$}.
$$
In addition, if $\psi\in {\mathcal S}_0$, then
$$
\sum_{j=-\infty}^\infty \widehat{\Phi}_j(\xi)
\widehat{\psi}(\xi)=\widehat{\psi}(\xi) \quad\mbox{for any $\xi\in
{\Bbb R}^d $}.
$$
That is, for $\psi\in {\mathcal S}_0$,
$$
\sum_{j=-\infty}^\infty \Phi_j \ast \psi = \psi
$$
and hence
$$
\sum_{j=-\infty}^\infty \Phi_j \ast f = f, \qquad f\in {\mathcal S}_0'
$$
in the sense of weak-$\ast$ topology of ${\mathcal S}_0'$. For
notational convenience, we define
\begin{equation}\label{del1}
\Delta_j f = \Phi_j \ast f, \qquad j \in {\Bbb Z}.
\end{equation}
\begin{define}
For $s\in {\Bbb R}$ and $1\le p,q\le \infty$, the homogeneous Besov
space $\mathring{B}^s_{p,q}$ consists of $f\in {\mathcal S}_0' $
satisfying
$$
\|f\|_{\mathring{B}^s_{p,q}} \equiv \|2^{js} \|\Delta_j
f\|_{L^p}\|_{l^q} <\infty.
$$
\end{define}
\vspace{.1in} We now choose $\Psi\in {\mathcal S}$ such that
$$
\widehat{\Psi} (\xi) = 1 - \sum_{j=0}^\infty \widehat{\Phi}_j (\xi),
\quad \xi \in {\Bbb R}^d.
$$
Then, for any $\psi\in {\mathcal S}$,
$$
\Psi \ast \psi + \sum_{j=0}^\infty \Phi_j \ast \psi =\psi
$$
and hence
\begin{equation}\label{sf}
\Psi \ast f + \sum_{j=0}^\infty \Phi_j \ast f =f
\end{equation}
in ${\mathcal S}'$ for any $f\in {\mathcal S}'$. To define the inhomogeneous Besov space, we set
\begin{equation} \label{del2}
\Delta'_j f = \left\{
\begin{array}{ll}
0,&\quad \mbox{if}\,\,j\le -2, \\
\Psi\ast f,&\quad \mbox{if}\,\,j=-1, \\
\Phi_j \ast f, &\quad \mbox{if} \,\,j=0,1,2,\cdots.
\end{array}
\right.
\end{equation}
\begin{define}
The inhomogeneous Besov space $B^s_{p,q}$ with $1\le p,q \le \infty$
and $s\in {\Bbb R}$ consists of functions $f\in {\mathcal S}'$
satisfying
$$
\|f\|_{B^s_{p,q}} \equiv \|2^{js} \|\Delta'_j f\|_{L^p} \|_{l^q}
<\infty.
$$
\end{define}

\vskip .1in
Many frequently used function spaces are special cases of  Besov spaces. The following proposition lists some useful equivalence and embedding relations.
\begin{prop}
For any $s\in \mathbb{R}$,
$$
\mathring{H}^s \sim \mathring{B}^s_{2,2}, \quad H^s \sim B^s_{2,2}.
$$
For any $s\in \mathbb{R}$ and $1<q<\infty$,
$$
\mathring{B}^{s}_{q,\min\{q,2\}} \hookrightarrow \mathring{W}_{q}^s \hookrightarrow \mathring{B}^{s}_{q,\max\{q,2\}}.
$$
In particular, $\mathring{B}^{0}_{q,\min\{q,2\}} \hookrightarrow L^q \hookrightarrow \mathring{B}^{0}_{q,\max\{q,2\}}$.
\end{prop}

\vskip .1in
For notational convenience, we write $\Delta_j$ for
$\Delta'_j$. There will be no confusion if we keep in mind that
$\Delta_j$'s associated the homogeneous Besov spaces is defined in
(\ref{del1}) while those associated with the inhomogeneous Besov
spaces are defined in (\ref{del2}). Besides the Fourier localization operators $\Delta_j$, the partial sum $S_j$ is also a useful notation. For an integer $j$,
$$
S_j \equiv \sum_{k=-1}^{j-1} \Delta_k,
$$
where $\Delta_k$ is given by (\ref{del2}). For any $f\in \mathcal{S}'$, the Fourier transform of $S_j f$ is supported on the ball of radius $2^j$.

\vskip .1in
Bernstein's inequalities is a useful tool on Fourier localized functions and these inequalities trade integrability for derivatives. The following proposition provides Bernstein type inequalities for fractional derivatives.
\begin{prop}\label{bern}
Let $\alpha\ge0$. Let $1\le p\le q\le \infty$.
\begin{enumerate}
\item[1)] If $f$ satisfies
$$
\mbox{supp}\, \widehat{f} \subset \{\xi\in \mathbb{R}^d: \,\, |\xi|
\le K 2^j \},
$$
for some integer $j$ and a constant $K>0$, then
$$
\|(-\Delta)^\alpha f\|_{L^q(\mathbb{R}^d)} \le C_1\, 2^{2\alpha j +
j d(\frac{1}{p}-\frac{1}{q})} \|f\|_{L^p(\mathbb{R}^d)}.
$$
\item[2)] If $f$ satisfies
\begin{equation*}\label{spp}
\mbox{supp}\, \widehat{f} \subset \{\xi\in \mathbb{R}^d: \,\, K_12^j
\le |\xi| \le K_2 2^j \}
\end{equation*}
for some integer $j$ and constants $0<K_1\le K_2$, then
$$
C_1\, 2^{2\alpha j} \|f\|_{L^q(\mathbb{R}^d)} \le \|(-\Delta)^\alpha
f\|_{L^q(\mathbb{R}^d)} \le C_2\, 2^{2\alpha j +
j d(\frac{1}{p}-\frac{1}{q})} \|f\|_{L^p(\mathbb{R}^d)},
$$
where $C_1$ and $C_2$ are constants depending on $\alpha,p$ and $q$
only.
\end{enumerate}
\end{prop}

\vskip .1in
Triebel-Lizorkin space is a direct generalization of Lebesgue spaces and Sobolev spaces.
\begin{define}
For any $s\in \mathbb{R}$, $0<p<\infty$ and $0<q\le \infty$, the Triebel-Lizorkin space $F^s_{p,q}$ is defined by
$$
F^s_{p,q} =\left\{f \in \mathcal{S}': \|f\|_{F^s_{p,q}} \equiv \left\| \|2^{sj} \Delta_j f \|_{l^q} \right\|_{L^p}  < \infty \right\}.
$$
\end{define}

\vskip .1in
Some ``classical" function spaces are special cases of the Triebel-Lizorkin spaces.
\begin{prop}
We have the following equivalence relations.
\begin{enumerate}
\item[(1)] For any $1<p<\infty$, $L^p\sim F^0_{p,2}$;
\item[(2)] For any $1<p<\infty$ and $m=1,2,\cdots$, $W^m_p \sim F^m_{p,2}$;
\item[(3)] For any $1\le p<\infty$ and $s\not=$ integers, $W^s_p \sim F^s_{p,p}$.
\end{enumerate}
\end{prop}

\vskip .2in
\section{Proof of Lemma \ref{triple}}
\label{prooflemma}

This appendix provides the proof of Lemma \ref{triple}.
\begin{proof}[Proof of Lemma \ref{triple}]
We first recall that, for any $q\ge 2$,
$$
\|f\|_{L^\infty(\mathbb{R})} \le \sqrt[q]{q}\, \|f\|_{L^{2(q-1)}(\mathbb{R})}^{1-\frac1q}\, \|f_x\|_{L^2(\mathbb{R})}^{\frac1q}.
$$
Therefore, by H\"{o}lder's inequality,
\begin{eqnarray}
&&\hskip-.3in
 \iint_{\mathbb{R}^2} | f \, g\, h|  \;dxdy \nonumber \\
&&
\leq C \int \left[ \left( \int |f|^2 \; dx \right)^{\frac12}
\left( \int |g|^2 \; dx \right)^{\frac12} \left( \sup_{-\infty < x < \infty} h \right) \right]   \;dy  \nonumber\\
&&
\leq C \int \left[ \left( \int |f|^2 \; dx \right)^{\frac12}
\left( \int |g|^2 \; dx \right)^{\frac12}  \left( \int |h|^{2(q-1)}\; dx \right)^{\frac1{2q}}
\left( \int |h_x|^2 \; dx \right)^{\frac1{2q}}  \right]   \;dy \nonumber\\
&&
\leq C\; \|f\|_2  \; \left(\int
\left( \int |g|^2 \; dx \right)^{\frac{q}{q-2}}dy \right)^{\frac{q-2}{2q}}\; \|h\|_{{2(q-1)}}^{1-\frac1q}\;
\|h_x\|_2^\frac1q. \label{e1}
\end{eqnarray}
By Minkowski's inequality,
\begin{eqnarray}
\left(\int
\left( \int |g|^2 \; dx \right)^{\frac{q}{q-2}}dy \right)^{\frac{q-2}{2q}}&\le& \left(\int \left( \int |g|^{\frac{2q}{q-2}} \; dy \right)^{\frac{q-2}{q}}dx \right)^{\frac12} \nonumber\\
&\le& \left(\int \left(\int g^2 dy\right)^{1-\frac1q} \left(\int g_y^2 dy\right)^{\frac1q} dx \right)^{\frac12} \nonumber\\
&\le& \|g\|_2^{1-\frac1q} \, \|g_y\|_2^{\frac1q}. \label{dd}
\end{eqnarray}
Inserting (\ref{dd}) in (\ref{e1}) yields the desired inequality. This completes the proof of Proposition \ref{triple}.
\end{proof}

\vskip .4in
\section*{Acknowledgements}
Cao is partially supported by NSF grant DMS 0709228 and a FIU foundation. Wu is partially supported by NSF grant DMS 0907913 and the Southwestern Bell Foundation at OSU.

\end{document}